\documentclass[a4paper,oneside,11pt]{article}
\usepackage{amsmath}
\usepackage{amssymb}
\usepackage{amsthm}

\usepackage{enumerate}
\usepackage{bm}
\usepackage[margin = 3cm]{geometry}
\usepackage{hyperref,url}
\usepackage{enumitem} 

\newcommand{\R}{\mathbb{R}}

\newcommand{\dHaus}{\, \mathrm{d} \Gamma}
\newcommand{\dd}{\, \mathrm{d} }
\newcommand{\dx}{\, \mathrm{dx}}
\newcommand{\dt}{\, \mathrm{dt}}
\newcommand{\ds}{\, \mathrm{ds}}
\newcommand{\pd}{\partial}

\newcommand{\pdnu}{ \pd_{\bm{\nu}}}
\newcommand{\abs}[1]{\left| #1 \right|}
\newcommand{\norm}[1]{\| #1 \|}
\newcommand{\bignorm}[1]{\left\| #1 \right\|}
\newcommand{\inner}[2]{\langle #1 , #2 \rangle}

\newcommand{\Laplace}{\Delta}
\newcommand{\mean}[1]{\overline{#1}}
\renewcommand{\div}{\, \mathrm{div}\,}
\renewcommand{\vec}{\bm}

\newtheorem{thm}{Theorem}[section]

\newtheorem{lemma}[thm]{Lemma}

\newtheorem{remark}{Remark}[section]

\newtheorem{defn}{Definition}[section]
\newtheorem{assump}{Assumption}[section]
\numberwithin{equation}{section}

\begin{document}

\title{On a Cahn--Hilliard--Darcy system for tumour growth with solution dependent source terms}

\author{Harald Garcke \footnotemark[1] \and Kei Fong Lam \footnotemark[1]}

\date{ }

\maketitle

\renewcommand{\thefootnote}{\fnsymbol{footnote}}
\footnotetext[1]{Fakult\"at f\"ur Mathematik, Universit\"at Regensburg, 93040 Regensburg, Germany
({\tt \{Harald.Garcke, Kei-Fong.Lam\}@mathematik.uni-regensburg.de}).}

\begin{abstract}
We study the existence of weak solutions to a mixture model for tumour growth that consists of a Cahn--Hilliard--Darcy system coupled with an elliptic reaction-diffusion equation.    The Darcy law gives rise to an elliptic equation for the pressure that is coupled to the convective Cahn--Hilliard equation through convective and source terms.    Both Dirichlet and Robin boundary conditions are considered for the pressure variable, which allows for the source terms to be dependent on the solution variables.
\end{abstract}

\noindent \textbf{Key words. } Cahn--Hilliard--Darcy system; phase field model; reaction-diffusion equation; tumour growth; chemotaxis; weak solutions; elliptic-parabolic system. \\

\noindent \textbf{AMS subject classification. } 35D30, 35Q35, 35Q92, 35K57, 76S05, 92C17, 92B05.

\section{Introduction}
At the fundamental level, cancer involves the unregulated growth of tissue inside the human body, which are caused by many biological and chemical mechanisms that take place at multiple spatial and temporal scales.  In order to understand how these multiscale mechanisms are driving the progression of the cancer cells, whose dynamics may be too complex to be approached by experimental techniques, mathematical modelling can be used to provide a tractable description of the dynamics that isolate the key mechanisms  and guide specific experiments.  

We focus on the subclass of models for tumour growth known as diffuse interface models.  These are continuum models that capture the macroscopic dynamics of the morphological changes of the tumour.  For the simplest situation where there are only tumour cells and host cells in the presence of a nutrient, the model equations consists of a Cahn--Hilliard equation coupled to a reaction-diffusion equation for the nutrient.  By treating the tumour and host cells as inertia-less fluids, a Darcy system can be appended to the Cahn--Hilliard equation, leading to a Cahn--Hilliard--Darcy system.  For details regarding the diffuse interface models for tumour growth we refer the reader to \cite{Chen,CLLW,book:Cristini,GLSS,Hawkins,OHP} and the references therein.

Our interest lies in providing analytical results for these models, namely in establishing the existence of weak solutions to the model equations.  Below, we introduce the Cahn--Hilliard--Darcy model to be studied:  Let $\Omega \subset \R^{d}$, $d = 2,3$, be a bounded domain with boundary $\Gamma$, and denote, for $T > 0$, $Q := \Omega \times (0,T)$ and $\Sigma := \Gamma \times (0,T)$.  We study the following elliptic-parabolic system:
\begin{subequations}\label{CHDN}
\begin{alignat}{3}
\div \vec{v} & = \Gamma_{\vec{v}}(\varphi, \sigma) && \text{ in } Q, \label{CHDN:div}   \\
\pd_{t} \varphi + \div (\varphi \vec{v}) & = \div (m(\varphi) \nabla \mu) + \Gamma_{\varphi}(\varphi, \sigma) && \text{ in } Q, \label{CHDN:varphi} \\
\mu & = A \Psi'(\varphi) - B \Laplace \varphi - \chi \sigma && \text{ in } Q, \label{CHDN:mu} \\
0 & = \Laplace \sigma - h(\varphi) \sigma && \text{ in } Q, \label{CHDN:sigma}  \\
\pdnu \varphi  & = 0, \quad \sigma = 1 && \text{ on } \Sigma, \label{CHDN:bdy} \\
\varphi(0) & = \varphi_{0} && \text{ in } \Omega, \label{CHDN:initial}
\end{alignat}
\end{subequations}
where $\pdnu f := \nabla f \cdot \vec{n}$ is the normal derivative of $f$ on the boundary $\Gamma$, with unit normal $\vec{n}$, and in this work, we focus on the following variants of Darcy's law and the boundary conditions
\begin{subequations}\label{DarcySystems}
\begin{alignat}{5}
\vec{v} &= -K (\nabla q + \varphi \nabla (\mu + \chi \sigma)) && \text{ in } Q, \quad  q = 0,&& \quad m(\varphi)\pdnu \mu = \varphi \vec{v} \cdot \vec{n} && \text{ on } \Sigma, \label{Dirichlet} \\
\vec{v} &= - K (\nabla p - (\mu + \chi \sigma) \nabla \varphi) && \text{ in } Q, \quad \mu = 0, &&\quad K \pdnu p = a(g-p) && \text{ on } \Sigma, \label{Robin} \\
\vec{v} & = -K (\nabla p - (\mu + \chi \sigma)\nabla \varphi) && \text{ in } Q, \quad \pdnu \mu = 0,&& \quad K \pdnu p = a(g-p) && \text{ on } \Sigma, \label{Neumann} 
\end{alignat}
\end{subequations}
for some positive constant $a$ and prescribed function $g$. In \eqref{CHDN}, $\vec{v}$ denotes the volume-averaged velocity of the cell mixture, $\sigma$ denotes the concentration of the nutrient, $\varphi$ denotes the difference in volume fractions, with $\{\varphi = 1\}$ representing unmixed tumour tissue, and $\{\varphi = -1\}$ representing the host tissue, and $\mu$ denotes the chemical potential for $\varphi$.

The positive constant $K$ is the permeability of the mixture, $m(\varphi)$ is a positive mobility for $\varphi$.  The constant parameter $\chi \geq 0$ regulates the chemotaxis effect (see \cite{GLSS} for more details), $\Psi'(\cdot)$ is the derivative of a potential function $\Psi(\cdot)$ that has two equal minima at $\pm 1$, $A$ and $B$ denote two positive constants related to the thickness of the diffuse interface and the surface tension, $h(\varphi)$ is an interpolation function that satisfies $h(-1) = 0$ and $h(1) = 1$.  

In \eqref{DarcySystems}, both $p$ and $q$ denote the pressure.  The Darcy law in \eqref{Dirichlet} with pressure $q$ can be obtained from the Darcy law in \eqref{Robin} and \eqref{Neumann} with pressure $p$ by setting $q = p - (\mu + \chi \sigma) \varphi$.  The source terms $\Gamma_{\vec{v}}$ and $\Gamma_{\varphi}$ model, for instance, the growth of the tumour and its effect on the velocity field.  We refer to \cite[\S 2.5]{GLSS} for a discussion regarding the choices for the source terms $\Gamma_{\varphi}, \Gamma_{\vec{v}}$.

We now compare the model \eqref{CHDN} with other models studied in the literature.
\begin{enumerate}
\item In the absence of velocity, i.e., setting $\vec{v} = 0$ in \eqref{CHDN:varphi} and neglecting \eqref{CHDN:div}, we obtain a elliptic-parabolic system that couples a Cahn--Hilliard equation with source term and an elliptic equation for the nutrient.  A similar system has been studied by the authors in \cite{GLDirichlet} with Dirichlet boundary conditions for $\varphi, \mu, \sigma$.  For systems where \eqref{CHDN:sigma} has an additional $\pd_{t}\sigma$ on the left-hand side, the well-posedness of solutions have been studied in \cite{Colli,Frigeri,GLNeumann,GLR} for particular choices of the source term $\Gamma_{\varphi}$.  We also mention the work of \cite{Dai} for the analysis of a system of equations similar to \eqref{CHDN} with $\chi = 0$.

\item In the case $\sigma = 0$, \eqref{CHDN} with the Darcy law \eqref{Robin} reduces to a Cahn--Hilliard--Darcy system, and well-posedness results have been established in \cite{LTZ} for $\Gamma_{\vec{v}} = \Gamma_{\varphi} = 0$ and $\pdnu p = \pdnu \mu = 0$ on $\Sigma$, and in \cite{JWZ} for prescribed source terms $\Gamma_{\vec{v}} = \Gamma_{\varphi} \neq 0$ and $\pdnu p = \pdnu \mu = 0$ on $\Sigma$.  In \cite{Bosia} a related system, known as the Cahn--Hilliard--Brinkman system, is studied, which features an additional term $-\nu \Laplace \vec{v}$ on the left-hand side of the Darcy law \eqref{Robin}, but with $\Gamma_{\vec{v}} = \Gamma_{\varphi} = 0$.  Analogously, \eqref{CHDN} without $\sigma$ and the Darcy law \eqref{Dirichlet} with boundary conditions $\pdnu p = \pdnu \mu = \pdnu \varphi = 0$ on $\Sigma$ has been studied in \cite{FengWise}.  For strong solutions to the Cahn--Hilliard--Darcy system on the $d$-dimensional torus, $d = 2,3$, we refer the reader to \cite{WangWu,WangZhang}.

\item In \cite{GLCHD}, the authors established the global existence of weak solutions to \eqref{CHDN} with the Darcy law \eqref{Robin} that features the following convection-reaction-diffusion equation for $\sigma$:
\begin{align*}
\pd_{t} \sigma + \div (\sigma \vec{v}) = \Laplace \sigma - \chi \Laplace \varphi - S,
\end{align*}
with a prescribed source term $\Gamma_{\vec{v}}$ and source terms $\Gamma_{\varphi}, S$ that depend on $\varphi, \sigma$ and $\mu$ with at most linear growth, along with the boundary conditions $\pdnu \mu = \pdnu \varphi = \pdnu p = 0$ and a Robin boundary condition for $\sigma$.  
\end{enumerate}

For the analyses performed on Cahn--Hilliard--Darcy systems in the literature, many have considered Neumann boundary conditions.  However, a feature of the Neumann conditions for $p$ and $\varphi$ is that 
\begin{align*}
\int_{\Omega} \Gamma_{\vec{v}} \dx = \int_{\Omega} \div \vec{v} \dx = \int_{\Gamma} \vec{v} \cdot \vec{n} \dHaus = \int_{\Gamma} -K \pdnu p + K(\mu + \chi \sigma) \pdnu \varphi \dHaus = 0,
\end{align*}
that is, the source term $\Gamma_{\vec{v}}$ necessarily has zero mean.  For source terms $\Gamma_{\vec{v}}$ that depend on $\varphi$ and $\sigma$, this property may not be satisfied in general.  To allow for source terms that need not have zero mean, one method is to prescribe alternate boundary conditions for the pressure, see for example \cite[\S 2.2.9]{book:ChenHuanMa} and \cite[\S 2.4.4]{GLSS}.  

In this work, we consider analysing the model with a Dirichlet boundary condition and also a Robin boundary condition for the pressure.  Then, the source term $\Gamma_{\vec{v}}$ does not need to fulfil the zero mean condition.  However, it turns out that in the derivation of a priori estimates for the model, we encounter the following:
\begin{itemize}
\item For the natural boundary condition $\pdnu \mu = 0$ and the Robin boundary condition $K \pdnu p = a(g-p)$ on $\Sigma$, we have to restrict our analysis to potentials $\Psi$ that have quadratic growth (Theorem \ref{thm:Neumann}).
\item To consider potentials with polynomial growth of order larger than two, we need to prescribe the boundary conditions \eqref{Dirichlet} and \eqref{Robin} for the chemical potential $\mu$ (Theorems \ref{thm:Dirichlet} and \ref{thm:Robin}).
\end{itemize}
Let us briefly motivate the choices in \eqref{Dirichlet} and \eqref{Robin}.  Due to the quasi-static nature of the nutrient equation \eqref{CHDN:sigma}, we do not obtain a natural energy identity for the system \eqref{CHDN} in contrast to the models studied in \cite{GLCHD,GLNeumann,GLSS}.  For simplicity, let $m(\varphi) = 1$, $K = 1$ and consider testing \eqref{CHDN:varphi} with $\mu + \chi \sigma$, \eqref{CHDN:mu} with $\pd_{t}\varphi$, the Darcy law \eqref{Robin} with $\vec{v}$.  Integrating by parts and upon adding leads to
\begin{equation}\label{Intro:Apriori:Robin}
\begin{aligned}
& \frac{\dd}{\dt} \int_{\Omega} A \Psi(\varphi ) + \frac{B}{2} \abs{\nabla \varphi}^{2} \dx + \int_{\Omega} \abs{\nabla \mu}^{2} + \abs{\vec{v}}^{2} \dx \\
& \quad = \int_{\Omega} - \chi \nabla \mu \cdot \nabla \sigma + \Gamma_{\vec{v}}(p - \varphi (\mu + \chi \sigma)) + \Gamma_{\varphi}(\mu + \chi \sigma) \dx \\
& \qquad + \int_{\Gamma} \pdnu \mu (\mu + \chi \sigma) - p \vec{v} \cdot \vec{n} \dHaus.
\end{aligned}
\end{equation}
If we prescribe the boundary conditions $\pdnu \mu = 0$ and $-\vec{v} \cdot \vec{n} = \pdnu p = a(g-p)$, i.e., the boundary conditions in \eqref{Neumann}, then the boundary term in \eqref{Intro:Apriori:Robin} poses no difficulties.  The main difficulty in obtaining a priori estimates from \eqref{Intro:Apriori:Robin} is to control the source terms $\Gamma_{\vec{v}} \mu \varphi$ and $\Gamma_{\varphi} \mu$ with the left-hand side of \eqref{Intro:Apriori:Robin}.  In the absence of any previous a priori estimates, to control terms involving $\mu$ by the term $\norm{\nabla \mu}_{L^{2}(\Omega)}^{2}$ on the left-hand side via the Poincar\'{e} inequality, an estimate of the square of the mean of $\mu$ is needed.  As observed in \cite{GLNeumann}, this leads to a restriction to quadratic growth assumptions for the potential $\Psi$.  

Furthermore, new difficulties arises in estimating the source term $\Gamma_{\vec{v}} p$ if we do not prescribe a Neumann boundary condition for $p$.  The methodology used in \cite{GLCHD,JWZ} to obtain an estimate for $\norm{p}_{L^{2}(\Omega)}$ relies on the assumption that $\Gamma_{\vec{v}}$ is prescribed and has zero mean, and $\pdnu p = 0$ on $\Sigma$.  The arguments in \cite{GLCHD,JWZ} seem not to be applicable for our present setting (see Remark \ref{Remark:WhyDirichlet} below), where $\Gamma_{\vec{v}}$ is dependent on $\varphi$ and $\sigma$, and a Robin boundary condition is prescribed for $p$.  This motivates the choice of a Dirichlet condition for $\mu$ to handle the source term $\Gamma_{\vec{v}} \varphi \mu$ and $\Gamma_{\varphi} \mu$, and as we will see later in Section \ref{sec:Robin} (specifically \eqref{Robin:Pressure:L2}), the Dirichlet boundary condition for $\mu$ is needed to obtain an $L^{2}$-estimate for $p$.

Alternatively, we may consider the discussion in \cite[\S 8]{GLCHD} regarding reformulations of the Darcy law.  Choosing $q = p - \varphi (\mu + \chi \sigma)$ leads to the Darcy law variant in \eqref{Dirichlet}.  A similar testing procedure leads to
\begin{equation}\label{Intro:Apriori:Dirichlet}
\begin{aligned}
& \frac{\dd}{\dt} \int_{\Omega} A \Psi(\varphi ) + \frac{B}{2} \abs{\nabla \varphi}^{2} \dx + \int_{\Omega} \abs{\nabla \mu}^{2} + \abs{\vec{v}}^{2} \dx \\
& \quad = \int_{\Omega} - \chi \nabla \mu \cdot \nabla \sigma + \Gamma_{\vec{v}}q + \Gamma_{\varphi}(\mu + \chi \sigma) \dx \\
& \qquad + \int_{\Gamma} (\pdnu \mu - \varphi \vec{v} \cdot \vec{n}) (\mu + \chi \sigma) - q \vec{v} \cdot \vec{n} \dHaus.
\end{aligned}
\end{equation}
Here we observed that the source term involving $\Gamma_{\vec{v}}$ simplifies to just $\Gamma_{\vec{v}} q$, and in exchange, we see the appearance of $(q + \varphi \mu + \chi \varphi \sigma) \vec{v} \cdot \vec{n}$ appearing in the boundary term.  Comparing to the previous set-up with \eqref{Robin}, we have shifted the problematic terms to the boundary integral.  Choosing $\vec{v} \cdot \vec{n} = 0$ on $\Sigma$ is not desirable, as equation \eqref{CHDN:div} would the imply that $\Gamma_{\vec{v}}(\varphi, \sigma)$ must have zero mean.  We may instead consider the boundary conditions
\begin{align*}
\pdnu \mu = 0, \quad \vec{v} \cdot \vec{n} = -\pdnu q  -  \chi \varphi \pdnu \sigma = a(q + \varphi (\mu + \chi \sigma)) \text{ on } \Sigma,
\end{align*}
then the boundary term in \eqref{Intro:Apriori:Dirichlet} poses no additional difficulties in obtaining a priori estimate.  In exchange, obtaining an estimate for $\norm{q}_{L^{2}(\Omega)}$ to deal with the source term $\Gamma_{\vec{v}}q$ becomes more involved, as the variational formulation for the pressure system now reads as 
\begin{align*}
\int_{\Omega} \nabla q \cdot \nabla \zeta \dx + \int_{\Gamma} a q \zeta \dHaus = \int_{\Omega} \Gamma_{\vec{v}} \zeta  - \varphi \nabla (\mu + \chi \sigma) \cdot \nabla \zeta \dx - \int_{\Gamma} a \varphi (\mu + \chi \sigma) \zeta \dHaus
\end{align*}
for a test function $\zeta$.  Estimates for $q$ will now involve an estimate for $\norm{\varphi \mu}_{L^{2}(\Gamma)}$, and this is more difficult to control than $\norm{\varphi \mu}_{L^{2}(\Omega)}$.  This motivates the choice of a Dirichlet condition for $q$ and the boundary condition $\pdnu \mu = \varphi \vec{v} \cdot \vec{n}$ to eliminate the boundary term in \eqref{Intro:Apriori:Dirichlet}.

This paper is organized as follows.  In Section \ref{sec:main} we state the main assumptions and the main results.  In Section \ref{sec:Dirichlet:Pressure} we outline the existence proof by first studying a parabolic-regularized variant of \eqref{CHDN}-\eqref{Dirichlet} where we add $\theta \pd_{t} \sigma$ to the left-hand side of \eqref{CHDN:sigma} for $\theta \in (0,1]$ and replace $\sigma$ with $\mathcal{T}(\sigma)$ in $\Gamma_{\vec{v}}$ and $\Gamma_{\varphi}$, where $\mathcal{T}$ is a cut-off operator.  The a priori estimates necessary for a Galerkin approximation to the parabolic-regularized problem is then derived, with which the weak existence for the original problem can be attained by passing to the limit $\theta \to 0$.  The analogous a priori estimates for the Robin boundary conditions \eqref{Robin} and \eqref{Neumann} are specified in Sections \ref{sec:Robin} and \ref{sec:quadratic}, respectively.

\smallskip

\textbf{Notation.}
For convenience, we will often use the notation $L^{p} := L^{p}(\Omega)$ and $W^{k,p} := W^{k,p}(\Omega)$ for any $p \in [1,\infty]$, $k > 0$ to denote the standard Lebesgue spaces and Sobolev spaces equipped with the norms $\norm{\cdot}_{L^{p}}$ and $\norm{\cdot}_{W^{k,p}}$.  In the case $p = 2$ we use $H^{k} := W^{k,2}$ and the norm $\norm{\cdot}_{H^{k}}$.  Due to the Dirichlet boundary condition for $\sigma$ and $\mu$, we denote the space $H^{1}_{0}$ as the completion of $C^{\infty}_{c}(\Omega)$ with respect to the $H^{1}$ norm.  We will use the isometric isomorphism $L^{p}(Q) \cong L^{p}(0,T;L^{p})$ and $L^{p}(\Sigma) \cong L^{p}(0,T;L^{p}(\Gamma))$ for any $p \in [1,\infty)$.  Moreover, the dual space of a Banach space $X$ will be denoted by $X^{*}$, and the duality pairing between $X$ and $X^{*}$ is denoted by $\inner{\cdot}{\cdot}_{X}$.  We denote the dual space to $H^{1}_{0}$ as $H^{-1}$.  For $d = 2$ or $3$, let $\dHaus$ denote integration with respect to the $(d-1)$ dimensional Hausdorff measure on $\Gamma$, and we denote $\R^{d}$-valued functions in boldface.  For convenience, we will often use the notation
\begin{align*}
\int_{Q} f := \int_{0}^{T} \int_{\Omega} f \dx \dt , \quad
\int_{\Omega_{t}} f := \int_{0}^{t} \int_{\Omega} f \dx \ds, \quad \int_{\Gamma_{t}} f := \int_{0}^{t} \int_{\Gamma} f \dHaus \ds
\end{align*}
for any $f \in L^{1}(Q)$ and for any $t \in (0,T]$.

\smallskip

\textbf{Useful preliminaries.}
For convenience, we recall the \emph{Poincar\'{e} inequality}: There exist a positive constant $C_{p}$  depending only on $\Omega$ such that
\begin{align}\label{Poincare}
\bignorm{f - \overline{f}}_{L^{r}} & \leq C_{p} \norm{\nabla f}_{L^{r}} \text{ for all } f \in W^{1,r}, 1 \leq r \leq \infty,
\end{align}
where $\mean{f} := \frac{1}{\abs{\Omega}} \int_{\Omega} f \dx$ denotes the mean of $f$.  Furthermore, we have
\begin{alignat}{3}
\norm{f}_{L^{2}} & \leq C_{p} \left ( \norm{\nabla f}_{L^{2}} + \norm{f}_{L^{2}(\Gamma)} \right ) && \text{ for } f \in H^{1}, \label{Poincare:Robin} \\
\norm{f}_{L^{2}} & \leq C_{p} \norm{\nabla f}_{L^{2}} && \text{ for } f \in H^{1}_{0}\label{Poincare:H10} .
\end{alignat}
The \emph{Gagliardo--Nirenberg interpolation inequality} in dimension $d$ (see \cite[Theorem 2.1]{book:DiBenedetto} and \cite[Theorem 5.8]{book:AdamsFournier}):  Let $\Omega$ be a bounded domain with Lipschitz boundary, and $f \in W^{m,r} \cap L^{q}$, $1 \leq q,r \leq \infty$.  For any integer $j$, $0 \leq j < m$, suppose there is $\alpha \in \R$ such that
\begin{align*}
\frac{1}{p} = \frac{j}{d} + \left ( \frac{1}{r} - \frac{m}{d} \right ) \alpha + \frac{1-\alpha}{q}, \quad \frac{j}{m} \leq \alpha \leq 1.
\end{align*}
If $r \in (1,\infty)$ and $m-j - \frac{d}{r}$ is a non-negative integer, we in addition assume $\alpha \neq 1$.  Under these assumptions, there exists a positive constant $C$ depending only on $\Omega$, $m$, $j$, $q$, $r$, and $\alpha$ such that
\begin{align}
\label{GagNirenIneq}
\norm{D^{j} f}_{L^{p}} \leq C \norm{f}_{W^{m,r}}^{\alpha} \norm{f}_{L^{q}}^{1-\alpha} .
\end{align}
For $f \in L^{2}$, $g \in L^{2}(\Gamma)$, and $\beta > 0$, let $u \in H^{1}$, $w \in H^{1}_{0}$ be the unique solutions to the elliptic problems
\begin{equation*}
\begin{alignedat}{3}
- \Laplace w = f & \text{ in } \Omega, \quad w = 0 &&\text{ on } \Gamma, \\
- \Laplace u = f &\text{ in } \Omega, \quad  \pdnu u + \beta u = g &&\text{ on } \Gamma.
\end{alignedat}
\end{equation*}
We use the notation $u = (-\Laplace_{R})^{-1}(f,\beta, g)$ and $w = (-\Laplace_{D})^{-1}(f)$.  Furthermore, if in addition $g \in H^{\frac{1}{2}}(\Gamma)$ and $\Gamma$ is a $C^{2}$-boundary, then by elliptic regularity theory \cite[Thm. 2.4.2.6]{Grisvard} and \cite[Thm. 2.4.2.5]{Grisvard}, it holds that $w \in H^{2} \cap H^{1}_{0}$ and $u \in H^{2}$ with
\begin{align*}
\norm{w}_{H^{2}} \leq C \norm{f}_{L^{2}}, \quad \norm{u}_{H^{2}} \leq C \left ( \norm{f}_{L^{2}} + \norm{g}_{H^{\frac{1}{2}}(\Gamma)} \right ).
\end{align*}

\section{Assumptions and main results}\label{sec:main}
\begin{assump}\label{assump:Main}
\
\begin{enumerate}[label=$(\mathrm{A \arabic*})$, ref = $\mathrm{A \arabic*}$]
\item \label{assump:Initial} $\Omega \subset \R^{d}$, $d = 2,3$, is a bounded domain with $C^{3}$-boundary $\Gamma$.  The positive constants $a, T, A, B, \chi, K$ are fixed.  The function $g \in L^{2}(\Sigma)$ and the initial condition $\varphi_{0} \in H^{1}$ are prescribed.
\item \label{assump:m:h} The mobility $m \in C^{0}(\R)$ satisfies $0 < m_{0} \leq m(s) \leq m_{1}$ for all $s \in \R$.  The function $h \in C^{0}(\R)$ is non-negative and is bounded above by 1.
\item \label{assump:Potential} The potential $\Psi \in C^{2}(\R)$ is non-negative and, for $ r \in [0,2]$ and for all $s \in \R$, there exist positive constants $C_{1}$, $C_{2}$, $C_{3}$ and $C_{4}$ such that
\begin{align*}
\Psi(s) \geq C_{1} \abs{s}^{2} - C_{2}, \quad \abs{\Psi''(s)} \leq C_{3} \left ( 1 + \abs{s}^{r} \right ), \quad \abs{\Psi'(s)} \leq C_{4} \left ( 1 + \Psi(s) \right ).
\end{align*}
\item \label{assump:Source} The source terms $\Gamma_{\vec{v}}$ and $\Gamma_{\vec{\varphi}}$ are of the form
\begin{align*}
\Gamma_{\vec{v}}(\varphi, \sigma) = b_{\vec{v}}(\varphi) \sigma + f_{\vec{v}}(\varphi), \quad \Gamma_{\varphi}(\varphi, \sigma) = b_{\varphi}(\varphi) \sigma + f_{\varphi}(\varphi),
\end{align*}
where $b_{\vec{v}}, b_{\varphi}, f_{\vec{v}}, f_{\varphi}$ are bounded  and continuous functions.
\end{enumerate}
\end{assump}

We first give the results to the problem \eqref{CHDN}, \eqref{Dirichlet}.

\begin{defn}\label{defn:Weaksoln:Dirichlet}
We call a quintuple $(\varphi, \mu, \sigma, \vec{v}, q)$ a weak solution to \eqref{CHDN}, \eqref{Dirichlet} if
\begin{align*}
\varphi & \in L^{\infty}(0,T;H^{1}) \cap L^{2}(0,T;H^{3}) \cap W^{1,\frac{8}{5}}(0,T;(H^{1})^{*}), \quad \vec{v} \in L^{2}(Q), \\
\sigma & \in (1 + L^{2}(0,T;H^{1}_{0})), \quad \mu \in L^{2}(0,T;H^{1}), \quad q  \in L^{\frac{8}{5}}(0,T;H^{1}_{0}),
\end{align*}
and satisfies $\varphi(0) = \varphi_{0}$, $0 \leq \sigma \leq 1$ a.e. in $Q$, and
\begin{subequations}\label{Weakform:Dirichlet}
\begin{alignat}{3}
0 & = \inner{\pd_{t}\varphi}{\zeta}_{H^{1}} + \int_{\Omega} m(\varphi) \nabla \mu \cdot \nabla \zeta - \varphi \vec{v} \cdot \nabla \zeta - \Gamma_{\varphi}(\varphi, \sigma) \zeta \dx, \label{Weak:D:varphi} \\
0 & = \int_{\Omega} (\mu + \chi \sigma) \zeta - A \Psi'(\varphi) \zeta - B \nabla \varphi \cdot \nabla \zeta \dx, \label{Weak:D:mu} \\
0 & = \int_{\Omega} \nabla \sigma \cdot \nabla \xi + h(\varphi) \sigma \xi \dx, \label{Weak:D:sigma} \\
0 & = \int_{\Omega} K \nabla q \cdot \nabla \xi - \Gamma_{\vec{v}}(\varphi, \sigma) \xi + K \varphi \nabla (\mu + \chi \sigma) \cdot \nabla \xi \dx, \label{Weak:D:Darcy} \\
0 & = \int_{\Omega} \vec{v} \cdot \vec{y} + K \nabla q \cdot \vec{y} + K \varphi \nabla (\mu + \chi \sigma) \cdot \vec{y} \dx, \label{Weak:D:velo} 
\end{alignat}
\end{subequations}
for a.e. $t \in (0,T)$ and all $\zeta \in H^{1}$, $\xi \in H^{1}_{0}$, $\vec{y} \in L^{2}$.
\end{defn}

\begin{thm}\label{thm:Dirichlet}
Under Assumption \ref{assump:Main}, there exists a weak solution to \eqref{CHDN}, \eqref{Dirichlet} in the sense of Definition \ref{defn:Weaksoln:Dirichlet}.
\end{thm}

For the problem \eqref{CHDN}, \eqref{Robin} we have the following.

\begin{defn}\label{defn:Weaksoln:Robin}
We call a quintuple $(\varphi, \mu, \sigma, \vec{v}, p)$ a weak solution to \eqref{CHDN}, \eqref{Robin} if
\begin{align*}
\varphi & \in L^{\infty}(0,T;H^{1}) \cap L^{2}(0,T;H^{3}) \cap W^{1,\frac{8}{5}}(0,T;H^{-1}), \quad \vec{v} \in L^{2}(Q), \\
\sigma & \in (1 + L^{2}(0,T;H^{1}_{0})), \quad \mu \in L^{2}(0,T;H^{1}_{0}), \quad p  \in L^{\frac{8}{5}}(0,T;H^{1}), \; p \vert_{\Sigma} \in L^{2}(\Sigma),
\end{align*}
and satisfies $\varphi(0) = \varphi_{0}$, $0 \leq \sigma \leq 1$ a.e. in $Q$, \eqref{Weak:D:mu}, \eqref{Weak:D:sigma}, and
\begin{subequations}\label{Weakform:Robin}
\begin{alignat}{3}
0 & = \inner{\pd_{t}\varphi}{\xi}_{H^{1}_{0}} + \int_{\Omega} m(\varphi) \nabla \mu \cdot \nabla \xi - \varphi \vec{v} \cdot \nabla \xi - \Gamma_{\vec{v}}(\varphi, \sigma) \xi \dx, \label{Weak:R:varphi} \\
0 & = \int_{\Omega} K \nabla p \cdot \nabla \zeta - \Gamma_{\vec{v}}(\varphi, \sigma) \zeta - K(\mu + \chi \sigma) \nabla \varphi \cdot \nabla \zeta \dx + \int_{\Gamma} a (p-g) \zeta \dHaus, \label{Weak:R:Darcy} \\
0 & = \int_{\Omega} \vec{v} \cdot \vec{y} + K \nabla p \cdot \vec{y} - K (\mu + \chi \sigma) \nabla \varphi \cdot \vec{y} \dx, \label{Weak:R:velo} 
\end{alignat}
\end{subequations}
for a.e. $t \in (0,T)$ and all $\zeta \in H^{1}$, $\xi \in H^{1}_{0}$, $\vec{y} \in L^{2}$.
\end{defn}

\begin{thm}\label{thm:Robin}
Under Assumption \ref{assump:Main}, there exists a weak solution to \eqref{CHDN}, \eqref{Robin} in the sense of Definition \ref{defn:Weaksoln:Robin}.
\end{thm}

Analogously for the problem \eqref{CHDN}, \eqref{Neumann} we have the following.

\begin{defn}\label{defn:Weaksoln:Neumann}
We call a quintuple $(\varphi, \mu, \sigma, \vec{v}, p)$ a weak solution to \eqref{CHDN}, \eqref{Neumann} if
\begin{align*}
\varphi & \in L^{\infty}(0,T;H^{1}) \cap L^{2}(0,T;H^{3}) \cap W^{1,\frac{8}{5}}(0,T;(H^{1})^{*}), \quad \vec{v} \in L^{2}(Q), \\
\sigma & \in (1 + L^{2}(0,T;H^{1}_{0})), \quad \mu \in L^{2}(0,T;H^{1}), \quad p  \in L^{\frac{8}{5}}(0,T;H^{1}), \; p \vert_{\Sigma} \in L^{2}(\Sigma),
\end{align*}
and satisfies $\varphi(0) = \varphi_{0}$, $0 \leq \sigma \leq 1$ a.e. in $Q$, \eqref{Weak:D:mu}, \eqref{Weak:D:sigma}, \eqref{Weak:R:Darcy} and \eqref{Weak:R:velo} and 
\begin{align}
0 & = \inner{\pd_{t}\varphi}{\zeta}_{H^{1}} + \int_{\Omega} m(\varphi) \nabla \mu \cdot \nabla \zeta + \nabla \varphi \cdot \vec{v} \zeta + \Gamma_{\vec{v}}(\varphi, \sigma) \varphi \zeta - \Gamma_{\varphi}(\varphi, \sigma) \zeta \dx, \label{Weak:N:varphi}
\end{align}
for a.e. $t \in (0,T)$ and all $\zeta \in H^{1}$, $\xi \in H^{1}_{0}$, $\vec{y} \in L^{2}$.
\end{defn}

\begin{thm}\label{thm:Neumann}
Under Assumption \ref{assump:Main}, with \eqref{assump:Potential} replaced by \begin{align}\label{assump:quadratic}
\Psi(s) \geq C_{1} \abs{s}^{2} - C_{2}, \quad \abs{\Psi''(s)} \leq C_{3} \quad \forall s \in \R,
\end{align}
for some positive constants $C_{1}, C_{2}, C_{3}$ , there exists a weak solution to \eqref{CHDN}, \eqref{Neumann} in the sense of Definition \ref{defn:Weaksoln:Neumann}.
\end{thm}

We use the fact that $H^{1} \subset \subset L^{2} \subset (H^{1})^{*}$, $H^{1} \subset \subset L^{2} \subset H^{-1}$, and \cite[\S 8, Cor. 4]{Simon} to deduce that $\varphi \in C^{0}([0,T];L^{2})$ in all cases, and thus $\varphi(0)$ makes sense as a function in $L^{2}$.  This implies that the initial condition $\varphi_{0}$ is attained in all cases.


\section{Dirichlet boundary conditions for the pressure} \label{sec:Dirichlet:Pressure}
We show the existence of weak solutions to \eqref{CHDN}, \eqref{Dirichlet} by means of a Galerkin approximation, and first consider a regularisation of \eqref{CHDN}, \eqref{Dirichlet}, where \eqref{CHDN:sigma} is replaced with 
\begin{align}\label{Parabolic:Nutrient} 
\theta \pd_{t} \sigma - \Laplace \sigma  + h(\varphi) \sigma = 0 \text{ in } Q, \quad
\sigma  = 1 \text{ on } \Sigma, \quad \sigma(0) = \sigma_{0} \text{ in } \Omega
\end{align}
for some $\theta \in (0,1]$, and $\sigma_{0} \in L^{2}(\Omega)$.  Furthermore, we introduce a cut-off operator $\mathcal{T}(s) := \max(0, \min(1,s))$ and replace the source terms with
\begin{align*}
\Gamma_{\vec{v}}(\varphi, \sigma) = b_{\vec{v}}(\varphi) \mathcal{T}(\sigma) + f_{\vec{v}}(\varphi), \quad \Gamma_{\varphi}(\varphi, \sigma) = b_{\varphi}(\varphi) \mathcal{T}(\sigma) + f_{\varphi}(\varphi).
\end{align*}
The procedure is to first use a Galerkin approximation to deduce the existence of a weak solution quintuple $(\varphi^{\theta}, \mu^{\theta}, \sigma^{\theta}, \vec{v}^{\theta}, q^{\theta})$ to the regularized problem, and subsequently employ a weak comparison principle at the continuous level to show $0 \leq \sigma^{\theta} \leq 1$ a.e. in $Q$, so that the cut-off operator $\mathcal{T}$ can then be neglected.   Then, we pass to the limit $\theta \to 0$ to obtain the existence of a weak solution to \eqref{CHDN}, \eqref{Dirichlet}.

Below we will derive the necessary a priori estimates to prove existence of weak solutions to the regularized problem
\begin{subequations}\label{Reg:prob:1}
\begin{alignat}{3}
\div \vec{v} & = b_{\vec{v}}(\varphi) \mathcal{T}(\sigma) + f_{\vec{v}}(\varphi) && \text{ in } Q, \label{Reg:1:Div} \\
\vec{v} & = -K (\nabla q + \varphi \nabla (\mu + \chi \sigma)) && \text{ in } Q, \label{Reg:1:Darcy} \\
\pd_{t} \varphi + \div (\varphi \vec{v}) & = \div (m(\varphi) \nabla \mu) +  b_{\varphi}(\varphi) \mathcal{T}(\sigma) + f_{\varphi}(\varphi) && \text{ in } Q,  \label{Reg:1:varphi} \\
\mu & = A \Psi'(\varphi) - B \Laplace \varphi - \chi \sigma && \text{ in } Q,  \label{Reg:1:chem} \\
\theta \pd_{t} \sigma & = \Laplace \sigma - h(\varphi) \sigma && \text{ in } Q, \label{Reg:1:sigma} \\
\pdnu \varphi & = 0, \quad  m(\varphi) \pdnu\mu = \varphi \vec{v} \cdot \vec{n}, \quad q = 0, \quad \sigma = 1 && \text{ on } \Sigma, \\
\varphi(0) & = \varphi_{0}, \quad \sigma(0) = \sigma_{0} && \text{ in } \Omega,
\end{alignat}
\end{subequations}
with an initial condition $0 \leq \sigma_{0} \leq 1$ a.e. in $\Omega$.

\begin{lemma}\label{lem:Reg:problem:1}
Under Assumption \ref{assump:Main} and $0 \leq \sigma_{0} \leq 1$ a.e. in $\Omega$, for any $\theta \in (0,1]$, there exists a weak solution quintuple $(\varphi^{\theta}, \mu^{\theta}, \sigma^{\theta}, \vec{v}^{\theta}, q^{\theta})$ in the sense of Definition \ref{defn:Weaksoln:Dirichlet} with additionally $\sigma^{\theta} \in H^{1}(0,T;H^{-1})$, $\sigma^{\theta}(0) = \sigma_{0}$ a.e. in $\Omega$, and \eqref{Weak:D:sigma} is replaced by
\begin{align}\label{Reg:weak:D:sigma}
0 & = \inner{\theta \pd_{t} \sigma^{\theta}}{\xi}_{H^{1}_{0}} + \int_{\Omega} \nabla \sigma^{\theta} \cdot \nabla \xi + h(\varphi^{\theta}) \sigma^{\theta} \xi \dx \quad \forall \xi \in H^{1}_{0}.
\end{align}
Furthermore, there exists a positive constant $C$ not depending on $\theta, \varphi^{\theta}, \mu^{\theta}, \sigma^{\theta}, \vec{v}^{\theta}, q^{\theta}$ such that \begin{equation}\label{Reg:problem:1:Est}
\begin{aligned}
& \norm{\Psi(\varphi^{\theta})}_{L^{\infty}(0,T;L^{1})} + \norm{\Psi'(\varphi^{\theta})}_{L^{2}(0,T;H^{1})} + \norm{\varphi^{\theta}}_{L^{\infty}(0,T;H^{1}) \cap L^{2}(0,T;H^{3})} \\
& \quad + \norm{\mu^{\theta}}_{L^{2}(0,T;H^{1})}  + \norm{\vec{v}^{\theta}}_{L^{2}(Q)} + \norm{q^{\theta}}_{L^{\frac{8}{5}}(0,T;H^{1}_{0})} + \norm{\pd_{t}\varphi^{\theta}}_{L^{\frac{8}{5}}(0,T;(H^{1})^{*})} \\
& \quad + \norm{\sigma^{\theta}}_{L^{2}(0,T;H^{1})} + \norm{\theta \pd_{t}\sigma^{\theta}}_{L^{2}(0,T;H^{-1})} \leq C.
\end{aligned}
\end{equation}

\end{lemma}

\begin{proof}
The details regarding the existence of Galerkin solutions via the theory of ODEs can be found in \cite{GLCHD, JWZ}, and so we will omit the details and focus only on the a priori estimates.  In the following, $C$ denotes a positive constant not depending on $(\varphi, \mu, \sigma, \vec{v}, q)$ and $\theta$, and may vary from line to line.

At the Galerkin level, we may replace duality pairings in \eqref{Weak:D:varphi} and \eqref{Reg:weak:D:sigma} with $L^{2}$-inner products.  For convenience let us reuse the variables $\varphi, \mu, \sigma, \vec{v}, q$ as the Galerkin solutions.  Let $Z > 0$ be a constant yet to be specified, then substituting $\xi = Z(\sigma - 1)$ in \eqref{Reg:weak:D:sigma}, $\zeta = \pd_{t}\varphi$ in \eqref{Weak:D:mu}, $\zeta = \mu + \chi \sigma$ in \eqref{Weak:D:varphi}, $\vec{y} = K^{-1} \vec{v}$ in \eqref{Weak:D:velo} and summing leads to
\begin{equation}\label{Apriori:Reg:D:Est}
\begin{aligned}
& \frac{\dd}{\dt} \int_{\Omega} A \Psi(\varphi) + \frac{B}{2} \abs{\nabla \varphi}^{2} + \frac{Z}{2} \theta \abs{\sigma -1}^{2} \dx \\
& \qquad  + \int_{\Omega} m(\varphi) \abs{\nabla \mu}^{2} + \frac{1}{K} \abs{\vec{v}}^{2} + Z \abs{\nabla \sigma}^{2} + Z h(\varphi) \abs{\sigma}^{2} \dx \\
& \quad = \int_{\Omega} - m(\varphi) \chi \nabla \mu \cdot \nabla \sigma + \Gamma_{\varphi} (\mu + \chi \sigma) + \Gamma_{\vec{v}} q + Z h(\varphi) \sigma \dx.
\end{aligned}
\end{equation}
Similarly to \cite{GLDirichlet} we estimate terms on the right-hand side involving $\sigma$ by $C_{1} \norm{\nabla \sigma}_{L^{2}}^{2} + C_{2} Z$ through the use of the Poincar\'{e} inequality, where $C_{1}, C_{2}$ are positive constants such that $C_{1}$ is independent of $Z$.  Thanks to the cutoff operator and the boundedness of $f_{\vec{v}}$ and $f_{\varphi}$, we see that
\begin{equation}\label{Dirichlet:Source:Est1}
\begin{aligned}
& \abs{\int_{\Omega} \Gamma_{\varphi} (\mu + \chi \sigma) + \Gamma_{\vec{v}} q \dx} \\
& \quad \leq C \left ( 1 + \norm{\mu - \mean{\mu}}_{L^{1}} + \abs{\mean{\mu}}_{L^{1}} + \norm{q}_{L^{2}} + \norm{\sigma - 1}_{L^{2}} \right ) \\
& \quad \leq C(1 + \abs{\mean{\mu}} + \norm{q}_{L^{2}}) + \frac{m_{0}}{4} \norm{\nabla \mu}_{L^{2}}^{2} + \norm{\nabla \sigma}_{L^{2}}^{2},
\end{aligned}
\end{equation}
where we have used the Poincar\'{e} inequality \eqref{Poincare} with $r = 1$ and Young's inequality.  From substituting $\zeta = 1$ in \eqref{Weak:D:mu} and using \eqref{assump:Potential}, we find that 
\begin{align}\label{Dirichlet:mean:mu}
\abs{\mean{\mu}} \leq C( 1 + \norm{\sigma - 1}_{L^{2}} + \norm{\Psi'(\varphi)}_{L^{1}}) \leq C \left ( 1 + \norm{\Psi(\varphi)}_{L^{1}} + \norm{\nabla \sigma}_{L^{2}} \right ).
\end{align}
To obtain an estimate of $\norm{q}_{L^{2}}$, we look at the pressure system, whose weak formulation is given by \eqref{Weak:D:Darcy}.  Let $f := (- \Laplace_{D})^{-1}(q/K)$, so that
\begin{align*}
\int_{\Omega} K \nabla f \cdot \nabla \phi \dx = \int_{\Omega} q \phi \dx \text{ for all } \phi \in H^{1}_{0}.
\end{align*}
Substituting $\xi = f$ in \eqref{Weak:D:Darcy} and $\phi = q$ in the above leads to
\begin{align*}
\norm{q}_{L^{2}}^{2} & = \int_{\Omega} K \nabla q \cdot \nabla f \dx = \int_{\Omega} \Gamma_{\vec{v}} f - K \varphi \nabla (\mu + \chi \sigma) \cdot \nabla f \dx \\
& \leq \norm{\Gamma_{\vec{v}}}_{L^{2}} \norm{f}_{L^{2}} + K \norm{\varphi \nabla (\mu + \chi \sigma)}_{L^{\frac{6}{5}}} \norm{\nabla f}_{L^{6}} \\
& \leq C \left ( 1 + \norm{\varphi}_{L^{3}} \norm{\nabla (\mu + \chi \sigma)}_{L^{2}} \right ) \norm{f}_{H^{2}}.
\end{align*}
Using the elliptic regularity estimate $\norm{f}_{H^{2}} \leq C \norm{q}_{L^{2}}$, we find that
\begin{equation}\label{Dirichlet:pressure:Est}
\begin{aligned}
\norm{q}_{L^{2}} & \leq C \left ( 1 + \norm{\varphi}_{H^{1}} \norm{\nabla (\mu + \chi \sigma)}_{L^{2}} \right ) \\
& \leq \frac{m_{0}}{4} \norm{\nabla \mu}_{L^{2}}^{2} + \norm{\nabla \sigma}_{L^{2}}^{2} + C \left ( 1 + \norm{\Psi(\varphi)}_{L^{1}} + \norm{\nabla \varphi}_{L^{2}}^{2} \right ),
\end{aligned}
\end{equation}
where we have used the Sobolev embedding $H^{1} \subset L^{3}$ and  \eqref{assump:Potential}.  Then, substituting the estimates \eqref{Dirichlet:mean:mu}, \eqref{Dirichlet:pressure:Est} into \eqref{Dirichlet:Source:Est1}, we find that the right-hand side of \eqref{Apriori:Reg:D:Est} can be estimated as
\begin{align*}
\abs{\mathrm{RHS}} & \leq \frac{m_{0}}{4} \norm{\nabla \mu}_{L^{2}}^{2} + \frac{\chi^{2} m_{1}^{2}}{m_{0}} \norm{\nabla \sigma}_{L^{2}}^{2} + Z \norm{\sigma - 1}_{L^{1}}  + Z \\
& \qquad + C \left ( 1 + \norm{\sigma-1}_{L^{2}} + \norm{\mu - \mean{\mu}}_{L^{1}} + \abs{\mean{\mu}} + \norm{q}_{L^{2}} \right ) \\
& \leq \frac{3 m_{0}}{4} \norm{\nabla \mu}_{L^{2}}^{2} + \left ( \frac{\chi^{2} m_{1}^{2}}{m_{0}} + 4 \right ) \norm{\nabla \sigma}_{L^{2}}^{2} \\
& \qquad +  C \left (1 + Z^{2} + \norm{\Psi(\varphi)}_{L^{1}} + \norm{\nabla \varphi}_{L^{2}}^{2} \right ).
\end{align*}
Neglecting the non-negative term $Z h(\varphi) \abs{\sigma}^{2}$ on the left-hand side of \eqref{Apriori:Reg:D:Est} and choosing $Z > \frac{\chi^{2} m_{1}^{2}}{m_{0}} + 4$ yields the differential inequality
\begin{align*}
& \frac{\dd}{\dt} \left ( \norm{\Psi(\varphi)}_{L^{1}} + \norm{\nabla \varphi}_{L^{2}}^{2} + \theta \norm{\sigma -1 }_{L^{2}}^{2} \right ) - C \left (\norm{\Psi(\varphi)}_{L^{1}} + \norm{\nabla \varphi}_{L^{2}}^{2} \right )\\
& \quad  + \norm{\nabla \mu}_{L^{2}}^{2} + \norm{\vec{v}}_{L^{2}}^{2} + \norm{\nabla \sigma}_{L^{2}}^{2} \leq C.
\end{align*}
By \eqref{assump:Initial}, \eqref{assump:Potential} and the Sobolev embedding $H^{1} \subset L^{6}$, it holds that $\Psi(\varphi_{0}) \in L^{1}$.  Hence, by an application of Gronwall's inequality we obtain
\begin{align*}
& \sup_{t \in (0,T]} \left ( \norm{\Psi(\varphi(t))}_{L^{1}} + \norm{\nabla \varphi(t)}_{L^{2}}^{2} + \theta \norm{\sigma(t) - 1}_{L^{2}}^{2} \right ) \\
& \quad + \norm{\nabla \mu}_{L^{2}(Q)}^{2} + \norm{\vec{v}}_{L^{2}(Q)}^{2} + \norm{\nabla \sigma}_{L^{2}(Q)}^{2} \leq C,
\end{align*}
where we have also used that $\theta \norm{\sigma_{0} - 1}_{L^{2}}^{2} \leq \norm{\sigma_{0} - 1}_{L^{2}}^{2}$ as $\theta \in (0,1]$.  Then, using \eqref{Dirichlet:mean:mu} and \eqref{assump:Potential} and the Poincar\'{e} inequality for $\varphi$ and $\mu$ yields
\begin{equation}\label{Dirichlet:Est:1}
\begin{aligned}
& \sup_{t \in (0,T]} \left ( \norm{\Psi(\varphi(t))}_{L^{1}} + \norm{\varphi(t)}_{H^{1}}^{2} + \theta \norm{\sigma}_{L^{2}}^{2} \right ) \\
& \qquad + \norm{\mu}_{L^{2}(0,T;H^{1})}^{2} + \norm{\vec{v}}_{L^{2}(Q)}^{2} + \norm{\sigma}_{L^{2}(0,T;H^{1})}^{2} \leq C.
\end{aligned}
\end{equation}
Next, looking at \eqref{Weak:D:mu} as an elliptic equation for $\varphi$, and using that the potential $\Psi'$ has at most polynomial growth of order 3, one can infer from the boundedness of $\varphi$ in $L^{\infty}(0,T;H^{1})$ that $\Psi'(\varphi)$ is bounded in $L^{2}(0,T;H^{1})$.  Then regularity theory gives $\varphi$ is bounded in $L^{2}(0,T;H^{3})$ and we obtain
\begin{align}\label{Dirichlet:Est:2}
\norm{\Psi'(\varphi)}_{L^{2}(0,T;H^{1})} +
\norm{\varphi}_{L^{2}(0,T;H^{3})} \leq C.
\end{align}
Then, substituting $\xi = q$ in \eqref{Weak:D:Darcy} and the Poincar\'{e} inequality \eqref{Poincare:H10} gives
\begin{align*}
K \norm{\nabla q}_{L^{2}}^{2} & \leq \norm{\Gamma_{\vec{v}}}_{L^{2}} \norm{q}_{L^{2}} + K \norm{\varphi \nabla (\mu + \chi \sigma)}_{L^{2}} \norm{\nabla q}_{L^{2}} \\
& \leq C + \frac{K}{2} \norm{\nabla q}_{L^{2}}^{2} + C \norm{\varphi}_{L^{\infty}}^{2} \norm{\nabla (\mu + \chi \sigma)}_{L^{2}}^{2} \\
& \leq C + \frac{K}{2} \norm{\nabla q}_{L^{2}}^{2} + C \norm{\varphi}_{L^{\infty}(0,T;L^{6})}^{\frac{3}{2}} \norm{\varphi}_{H^{3}}^{\frac{1}{2}}  \norm{\nabla (\mu + \chi \sigma)}_{L^{2}}^{2},
\end{align*}
where we have also used the Gagliardo--Nirenburg inequality \eqref{GagNirenIneq} in three dimensions.  Thus we obtain
\begin{equation}\label{Dirichlet:Est:3}
\begin{aligned}
\int_{0}^{T} \norm{q}_{H^{1}}^{\frac{8}{5}} \dt & \leq C \left ( 1 + \norm{\varphi}_{L^{\infty}(0,T;H^{1})}^{\frac{6}{5}} \int_{0}^{T} \norm{\varphi}_{H^{3}}^{\frac{2}{5}} \norm{\nabla (\mu + \chi \sigma)}_{L^{2}}^{\frac{8}{5}} \dt \right ) \\
& \leq C \left ( 1 + \norm{\varphi}_{L^{2}(0,T;H^{3})}^{\frac{2}{5}} \norm{\nabla (\mu + \chi \sigma)}_{L^{2}(Q)}^{\frac{8}{5}} \right ) \leq C.
\end{aligned}
\end{equation}
Lastly, we see that for any $\zeta \in L^{\frac{8}{3}}(0,T;H^{1})$,
\begin{equation}\label{convection:term}
\begin{aligned}
& \abs{\int_{Q} \varphi \vec{v} \cdot \nabla \zeta} \leq \int_{0}^{T} \norm{\varphi}_{L^{\infty}} \norm{\vec{v}}_{L^{2}} \norm{\nabla \zeta}_{L^{2}} \dt \\
& \quad \leq C \norm{\varphi}_{L^{\infty}(0,T;H^{1})}^{\frac{3}{4}} \norm{\vec{v}}_{L^{2}(Q)} \norm{\varphi}_{L^{2}(0,T;H^{3})}^{\frac{1}{4}} \norm{\zeta}_{L^{\frac{8}{3}}(0,T;H^{1})}  \leq C \norm{\zeta}_{L^{\frac{8}{3}}(0,T;H^{1})},
\end{aligned}
\end{equation}
and so from \eqref{Weak:D:varphi}, we obtain
\begin{align}\label{Dirichlet:Est:4}
\norm{\pd_{t}\varphi}_{L^{\frac{8}{5}}(0,T;(H^{1})^{*})} \leq C \left ( 1 + \norm{\nabla \mu}_{L^{2}(Q)} + \norm{\div(\varphi \vec{v})}_{L^{\frac{8}{5}}(0,T;(H^{1})^{*})} \right ) \leq C.
\end{align}
Similarly, from \eqref{Reg:weak:D:sigma} we see that
\begin{align}\label{Dirichlet:Est:5}
\norm{\theta \pd_{t} \sigma}_{L^{2}(0,T;H^{-1})} \leq \norm{\sigma}_{L^{2}(0,T;H^{1})} \leq C.
\end{align}
The a priori estimates \eqref{Dirichlet:Est:1}, \eqref{Dirichlet:Est:2}, \eqref{Dirichlet:Est:3}, \eqref{Dirichlet:Est:4} and \eqref{Dirichlet:Est:5} are sufficient to deduce the existence of a weak solution quadruple $(\varphi^{\theta}, \mu^{\theta}, \sigma^{\theta}, \vec{v}^{\theta}, q^{\theta})$ to \eqref{Reg:prob:1} with the regularities stated in Lemma \ref{lem:Reg:problem:1} which satisfies \eqref{Weak:D:varphi}, \eqref{Weak:D:mu}, \eqref{Weak:D:Darcy}, \eqref{Weak:D:velo}, and \eqref{Reg:weak:D:sigma} for a.e. $t \in (0,T)$ and all $\zeta \in H^{1}$, $\vec{y} \in L^{2}$, $\xi \in H^{1}_{0}$.  We refer the reader to \cite{GLCHD} for the details in passing to the limit.  Let us just mention that thanks to boundedness in $L^{2}(0,T;H^{1}_{0}) \cap H^{1}(0,T;H^{-1})$ and \cite[\S 8, Cor. 4]{Simon} the Galerkin approximations for $\sigma$ converges strongly in $L^{2}(Q)$ and hence also a.e. in $Q$.  Furthermore, the estimate \eqref{Reg:problem:1:Est} is obtained by passing to the limit in the a priori estimates \eqref{Dirichlet:Est:1}, \eqref{Dirichlet:Est:2}, \eqref{Dirichlet:Est:3}, \eqref{Dirichlet:Est:4} and \eqref{Dirichlet:Est:5} for the Galerkin approximation and using weak/weak* lower semi-continuity of the norms.

To complete the proof, it remains to show that $0 \leq \sigma^{\theta} \leq 1$ a.e. in $Q$ by means of a weak comparison principle.  For this we substitute $\xi = (\sigma^{\theta} - 1)_{+} := \max (\sigma^{\theta} - 1, 0)$ and $\xi = (\sigma^{\theta})_{-} := \max (- \sigma^{\theta}, 0)$ in \eqref{Reg:weak:D:sigma}, and note that due to the boundary condition $\sigma^{\theta} = 1$ on $\Sigma$, necessarily $(\sigma^{\theta} - 1)_{+}, (\sigma^{\theta})_{-} \in H^{1}_{0}$.  The former yields
\begin{align*}
& \frac{\theta}{2} \frac{\dd}{\dt} \norm{(\sigma^{\theta} - 1)_{+}}_{L^{2}}^{2}  \\
& \quad = - \norm{\nabla (\sigma^{\theta} - 1)_{+}}_{L^{2}}^{2} - \int_{\Omega} h(\varphi) \abs{(\sigma^{\theta} - 1)_{+}}^{2} + h(\varphi) (\sigma^{\theta} - 1)_{+} \dx  \leq 0,
\end{align*}
and the latter yields
\begin{align*}
\frac{\theta}{2} \frac{\dd}{\dt} \norm{(\sigma^{\theta})_{-}}_{L^{2}}^{2} = - \norm{\nabla (\sigma^{\theta})_{-}}_{L^{2}}^{2} - \int_{\Omega} h(\varphi) \abs{(\sigma^{\theta})_{-}}^{2} \dx \leq 0.
\end{align*}
From both inequalities we infer that for any $t \in (0,T)$,
\begin{align*}
\norm{(\sigma^{\theta}(t)-1)_{+}}_{L^{2}}^{2} \leq \norm{(\sigma_{0} - 1)_{+}}_{L^{2}}^{2} = 0, \quad \norm{(\sigma^{\theta}(t))_{-}}_{L^{2}}^{2} \leq \norm{(\sigma_{0})_{-}}_{L^{2}}^{2} = 0,
\end{align*}
as $0 \leq \sigma_{0} \leq 1$ a.e. in $\Omega$.  This yields that $0 \leq \sigma^{\theta} \leq 1$ a.e. in $Q$.
\end{proof}

At this point, we can neglect the cut-off operator $\mathcal{T}$ in \eqref{Reg:prob:1} and now pass to the limit $\theta \to 0$.  By virtue of \eqref{Reg:problem:1:Est} we have boundedness of $(\varphi^{\theta}, \mu^{\theta}, \sigma^{\theta}, \vec{v}^{\theta}, q^{\theta})$ in the Bochner spaces stated in Definition \ref{defn:Weaksoln:Dirichlet}.  Denoting the limit functions as $(\varphi, \mu, \sigma, \vec{v}, q)$, it is a standard argument to show that the above quintuple is a weak solution of \eqref{CHDN}-\eqref{Dirichlet} in the sense of Definition \ref{defn:Weaksoln:Dirichlet}, and thus we omit the details. 

\section{Robin boundary conditions for the pressure}\label{sec:Robin}
To prove Theorem \ref{thm:Robin} for the system \eqref{CHDN}-\eqref{Robin}, it suffices to prove the existence of a weak solution $(\varphi^{\theta}, \mu^{\theta}, \sigma^{\theta}, \vec{v}^{\theta}, p^{\theta})$ to the regularized problem consisting of \eqref{Reg:1:Div}, \eqref{Reg:1:varphi}, \eqref{Reg:1:chem}, \eqref{Reg:1:sigma} and 
\begin{align}\label{Reg:2:Darcy}
\vec{v} = - K \left ( \nabla p - \left ( \mu + \chi \sigma \right ) \nabla \varphi \right ) \; \text{ in } Q,
\end{align}
along with the initial-boundary conditions
\begin{align*}
\pdnu \varphi = 0, \quad \mu = 0, \quad K \pdnu p = a(g-p), \quad \sigma = 1 & \text{ on } \Sigma, \\
\varphi(0) = \varphi_{0}, \quad \sigma(0) = \sigma_{0} & \text{ in } \Omega,
\end{align*}
and then pass to the limit $\theta \to 0$.  We focus on obtaining a priori estimates for the regularized problem and omit the argument for $\theta \to 0$ as it follows straightforwardly from the a priori estimates.

\begin{lemma}\label{lem:Reg:problem:2}
Under Assumption \ref{assump:Main} and $0 \leq \sigma_{0} \leq 1$ a.e. in $\Omega$, for any $\theta \in (0,1]$, there exists a weak solution quintuple $(\varphi^{\theta}, \mu^{\theta}, \sigma^{\theta}, \vec{v}^{\theta}, p^{\theta})$ in the sense of Definition \ref{defn:Weaksoln:Robin} with additionally $\sigma^{\theta} \in H^{1}(0,T;H^{-1})$, $\sigma^{\theta}(0) = \sigma_{0}$ a.e. in $\Omega$, and \eqref{Weak:D:sigma} is replaced by \eqref{Reg:weak:D:sigma}.  Furthermore, there exists a positive constant $C$ not depending on $\theta, \varphi^{\theta}, \mu^{\theta}, \sigma^{\theta}, \vec{v}^{\theta}, p^{\theta}$ such that \begin{equation}\label{Reg:problem:2:Est}
\begin{aligned}
& \norm{\Psi(\varphi^{\theta})}_{L^{\infty}(0,T;L^{1})} + \norm{\Psi'(\varphi^{\theta})}_{L^{2}(0,T;H^{1})} + \norm{\varphi^{\theta}}_{L^{\infty}(0,T;H^{1}) \cap L^{2}(0,T;H^{3})} \\
& \quad + \norm{\mu^{\theta}}_{L^{2}(0,T;H^{1})}  + \norm{\vec{v}^{\theta}}_{L^{2}(Q)} + \norm{p^{\theta}}_{L^{\frac{8}{5}}(0,T;H^{1})} + \norm{\pd_{t}\varphi^{\theta}}_{L^{\frac{8}{5}}(0,T;H^{-1})} \\
& \quad + \norm{\sigma^{\theta}}_{L^{2}(0,T;H^{1})} + \norm{\theta \pd_{t}\sigma^{\theta}}_{L^{2}(0,T;H^{-1})} + \norm{p^{\theta}}_{L^{2}(\Sigma)} \leq C.
\end{aligned}
\end{equation}
\end{lemma}

\begin{proof}
Once again we will only derive the a priori estimates.  Substituting $\xi = Z(\sigma - 1)$ in \eqref{Reg:weak:D:sigma} for some constant $Z > 0$ yet to be determined, $\zeta = \pd_{t}\varphi$ in \eqref{Weak:D:mu}, $\xi = \mu + \chi (\sigma - 1)$ in \eqref{Weak:R:varphi}, $\vec{y} = K^{-1} \vec{v}$ in \eqref{Weak:R:velo}, and summing leads to
\begin{equation}\label{Aux:CHD:Robin:Est}
\begin{aligned}
& \frac{\dd}{\dt} \int_{\Omega} A \Psi(\varphi) + \frac{B}{2} \abs{\nabla \varphi}^{2} - \chi \varphi + \frac{Z}{2} \theta \abs{\sigma - 1}^{2} \dx  \\
& \quad  + \int_{\Omega} m(\varphi) \abs{\nabla \mu}^{2} + \frac{1}{K} \abs{\vec{v}}^{2} + Z \abs{\nabla \sigma}^{2} + Z h(\varphi) \abs{\sigma}^{2} \dx + a \norm{p}_{L^{2}(\Gamma)}^{2} \\ 
& \quad = \int_{\Omega} - \chi m(\varphi) \nabla \mu \cdot \nabla \sigma + \Gamma_{\varphi} (\mu + \chi (\sigma - 1)) + Z h(\varphi) \sigma \dx \\
& \qquad + \int_{\Omega} p \Gamma_{\vec{v}} + \varphi \vec{v} \cdot \nabla (\mu + \chi (\sigma -1 )) + (\mu + \chi \sigma) \nabla \varphi \cdot \vec{v} \dx + \int_{\Gamma} a g p \dHaus.
\end{aligned}
\end{equation}
Using that $(\mu + \chi (\sigma - 1)) = 0$ on $\Gamma$ and the product rule, we have
\begin{align*}
& \int_{\Omega} \varphi \vec{v} \cdot \nabla (\mu + \chi (\sigma - 1)) + (\mu + \chi \sigma) \nabla \varphi \cdot \vec{v} \dx \\
& \quad = \int_{\Omega} \chi \vec{v} \cdot \nabla \varphi - \Gamma_{\vec{v}} \varphi (\mu + \chi(\sigma - 1)) \dx.
\end{align*}
Thus, we obtain the following identity from integrating \eqref{Aux:CHD:Robin:Est} in time
\begin{equation}\label{Aux:CHD:Robin:Est:2}
\begin{aligned}
& \int_{\Omega} \left ( A \Psi(\varphi) + \frac{B}{2} \abs{\nabla \varphi}^{2} - \chi \varphi+ \frac{Z}{2} \theta \abs{\sigma - 1}^{2} \right )(t) \dx \\
& \qquad  + \int_{\Omega_{t}} \left ( m(\varphi) \abs{\nabla \mu}^{2} + \frac{1}{K} \abs{\vec{v}}^{2} + Z \abs{\nabla \sigma}^{2} + Z h(\varphi) \abs{\sigma}^{2} \right ) + \int_{\Gamma_{t}} a \abs{p}^{2} \\
& \quad = \int_{\Omega_{t}} \left ( - \chi m(\varphi) \nabla \mu \cdot \nabla \sigma + \chi \nabla \varphi \cdot \vec{v}  + Z h(\varphi) \sigma \right ) + \int_{\Gamma_{t}} a g p \\
& \qquad + \int_{\Omega_{t}} \left ( \Gamma_{\vec{v}} (p - \varphi (\mu + \chi (\sigma - 1))) + \Gamma_{\varphi}(\mu + \chi (\sigma - 1))  \right ) \\
& \qquad + \int_{\Omega} \left (A \Psi(\varphi_{0}) + \frac{B}{2} \abs{\nabla \varphi_{0}}^{2} - \chi \varphi_{0} + \frac{Z}{2} \theta \abs{\sigma_{0} - 1}^{2} \right ) \dx =: I_{1} + I_{2} + I_{3}.
\end{aligned}
\end{equation}
Note that by \eqref{assump:Potential} and the fact that $\theta \in (0,1]$, the third term $I_{3}$ on the right-hand side of \eqref{Aux:CHD:Robin:Est:2} is bounded, and by Young's inequality
\begin{align*}
\abs{\int_{\Omega} \chi \varphi \dx} \leq \chi \abs{\Omega}^{\frac{1}{2}} \norm{\varphi}_{L^{2}} \leq \frac{A}{2C_{1}} \norm{\varphi}_{L^{2}}^{2} + C \leq \frac{A}{2} \norm{\Psi(\varphi)}_{L^{1}} + C,
\end{align*}
which implies that
\begin{align*}
\int_{\Omega} \left ( A \Psi(\varphi) - \chi \varphi \right ) (t) \dx \geq \frac{A}{2} \norm{\Psi(\varphi(t))}_{L^{1}} - C.
\end{align*}
Next, for $I_{1}$, using the Poincar\'{e} inequality in $L^{1}$ on $(\sigma - 1)$, H\"{o}lder's inequality and Young's inequality, we have
\begin{align*}
\abs{I_{1}} & \leq \frac{m_{0}}{4} \norm{\nabla \mu}_{L^{2}(Q)}^{2} + \left ( \frac{\chi^{2} m_{1}^{2}}{m_{0}} + 1 \right ) \norm{\nabla \sigma}_{L^{2}(Q)}^{2} + \frac{1}{2K} \norm{\vec{v}}_{L^{2}(Q)}^{2} + \frac{a}{2} \norm{p}_{L^{2}(\Sigma)}^{2} \\
& \quad + C \left ( 1 + Z^{2} + \norm{\nabla \varphi}_{L^{2}(Q)}^{2} + \norm{g}_{L^{2}(\Sigma)}^{2} \right ).
\end{align*}
It remains to estimate $I_{2}$, and we first obtain an estimate on $\norm{p}_{L^{2}}$ by looking at the pressure system, whose weak formulation is given by \eqref{Weak:R:Darcy}.  Let $f := (-\Laplace_{R})^{-1}(p/K,a/K,0)$, so that
\begin{align*}
\int_{\Omega} K \nabla f \cdot \nabla \phi \dx + \int_{\Gamma} a f \phi \dHaus = \int_{\Omega} p \phi \dx \text{ for all } \phi \in H^{1}.
\end{align*}
Substituting $\zeta = f$ in \eqref{Weak:R:Darcy} and $\phi = p$ in the above leads to
\begin{equation}
\begin{aligned}
\norm{p}_{L^{2}}^{2} &  = \int_{\Omega} \Gamma_{\vec{v}} f + K (\mu + \chi \sigma) \nabla \varphi \cdot \nabla f \dx + \int_{\Gamma} a g f \dHaus \\
& \leq \norm{\Gamma_{\vec{v}}}_{L^{2}} \norm{f}_{L^{2}} + K \norm{(\mu + \chi \sigma) \nabla \varphi}_{L^{\frac{6}{5}}} \norm{\nabla f}_{L^{6}} + a \norm{g}_{L^{2}(\Gamma)} \norm{f}_{L^{2}(\Gamma)} \\
& \leq C \left ( 1 + \norm{g}_{L^{2}(\Gamma)} + \norm{(\mu + \chi \sigma) \nabla \varphi}_{L^{\frac{6}{5}}} \right ) \norm{f}_{H^{2}}.
\end{aligned}
\end{equation}
Using the elliptic regularity estimate $\norm{f}_{H^{2}} \leq C \norm{p}_{L^{2}}$, we obtain, analogous to \eqref{Dirichlet:pressure:Est},
\begin{equation}\label{Robin:Pressure:L2}
\begin{aligned}
\norm{p}_{L^{2}} & \leq C \left ( 1 + \norm{g}_{L^{2}(\Gamma)} + \norm{(\mu + \chi \sigma) \nabla \varphi}_{L^{\frac{6}{5}}} \right ) \\
& \leq C \left ( 1 + \norm{g}_{L^{2}(\Gamma)} + \norm{\mu + \chi \sigma}_{L^{6}} \norm{\nabla \varphi}_{L^{\frac{3}{2}}} \right ) \\
& \leq C \left ( 1 + \norm{g}_{L^{2}(\Gamma)} + \left (1+ \norm{\nabla \mu}_{L^{2}} + \norm{\nabla \sigma}_{L^{2}} \right ) \norm{\nabla \varphi}_{L^{\frac{3}{2}}} \right ),
\end{aligned}
\end{equation}
where we have applied the Poincar\'{e} inequality \eqref{Poincare:H10} to $\mu$ and $\sigma - 1$, and the Sobolev embedding $H^{1} \subset L^{6}$.  Using the boundedness of $\Gamma_{\vec{v}}$ and $\Gamma_{\varphi}$, \eqref{assump:Potential} and $\norm{p}_{L^{1}(Q)} \leq C \norm{p}_{L^{1}(0,T;L^{2})}$, we see that
\begin{align*}
& \abs{I_{2}} \leq C \left (1 + \norm{p}_{L^{1}(Q)} + \left ( 1 + \norm{\varphi}_{L^{2}(Q)} \right ) \left ( \norm{\mu}_{L^{2}(Q)} + \norm{\sigma - 1}_{L^{2}(Q)} \right ) \right ) \\
& \, \leq C \left ( 1 + \norm{g}_{L^{2}(\Sigma)} + \norm{\nabla \varphi}_{L^{2}(Q)}^{2} + \norm{\varphi}_{L^{2}(Q)}^{2} \right ) + \frac{m_{0}}{4} \norm{\nabla \mu}_{L^{2}(Q)}^{2} + \norm{\nabla \sigma}_{L^{2}(Q)}^{2} \\
& \, \leq C \left ( 1 + \norm{\Psi(\varphi)}_{L^{1}(Q)} + \norm{\nabla \varphi}_{L^{2}(Q)}^{2} + \norm{g}_{L^{2}(\Sigma)}^{2} \right )  + \frac{m_{0}}{4} \norm{\nabla \mu}_{L^{2}(Q)}^{2} + \norm{\nabla \sigma}_{L^{2}(Q)}^{2}.
\end{align*}
Thus, choosing $Z > \frac{\chi^{2} m_{1}^{2}}{m_{0}^{2}} + 2$, we obtain from \eqref{Aux:CHD:Robin:Est:2} the inequality
\begin{align*}
& \left ( \norm{\Psi(\varphi(t))}_{L^{1}} + \norm{\nabla \varphi(t)}_{L^{2}}^{2} + \theta \norm{\sigma(t) - 1}_{L^{2}}^{2} \right ) \\
& \qquad + \norm{\nabla \mu}_{L^{2}(Q)}^{2} + \norm{\vec{v}}_{L^{2}(Q)}^{2} + \norm{\nabla \sigma}_{L^{2}(Q)}^{2} + \norm{p}_{L^{2}(\Sigma)}^{2} \\
& \quad \leq C \left ( 1 + \norm{g}_{L^{2}(\Sigma)}^{2} + \norm{\Psi(\varphi)}_{L^{1}(Q)} + \norm{\nabla \varphi}_{L^{2}(Q)}^{2} \right ),
\end{align*}
for all $t \in (0,T]$.  Applying the integral version of Gronwall's inequality  \cite[Lem. 3.1]{GLNeumann}, we obtain
\begin{equation}\label{Robin:Est}
\begin{aligned}
& \sup_{t \in (0,T]} \left ( \norm{\Psi(\varphi(t))}_{L^{1}} + \norm{\nabla \varphi(t)}_{L^{2}}^{2} + \theta \norm{\sigma(t)-1}_{L^{2}}^{2} \right ) \\
& \quad + \norm{\nabla \mu}_{L^{2}(Q)}^{2} + \norm{\nabla \sigma}_{L^{2}(Q)}^{2} + \norm{\vec{v}}_{L^{2}(Q)}^{2} + \norm{p}_{L^{2}(\Sigma)}^{2} \leq C.
\end{aligned}
\end{equation}
Then, using \eqref{assump:Potential} and the Poincar\'{e} inequality for $\mu$ and $\sigma$, this yields
\begin{equation}\label{Robin:Est:1}
\begin{aligned}
& \sup_{t \in (0,T]} \left ( \norm{\Psi(\varphi(t))}_{L^{1}} + \norm{\varphi(t)}_{H^{1}}^{2} + \theta \norm{\sigma(t) -1}_{L^{2}}^{2} \right ) \\
& \quad  + \norm{\mu}_{L^{2}(0,T;H^{1})}^{2} + \norm{\sigma}_{L^{2}(0,T;H^{1})}^{2} + \norm{\vec{v}}_{L^{2}(Q)}^{2} + \norm{p}_{L^{2}(\Sigma)}^{2} \leq C.
\end{aligned}
\end{equation}
Analogous to the Dirichlet case, we have
\begin{align}\label{Robin:Est:2}
\norm{\Psi'(\varphi)}_{L^{2}(0,T;H^{1})} +
\norm{\varphi}_{L^{2}(0,T;H^{3})} \leq C.
\end{align}
Then, from \eqref{Weak:R:Darcy} and the Poincar\'{e} inequality \eqref{Poincare:Robin}, it holds that
\begin{align*}
& K \norm{\nabla p}_{L^{2}}^{2} + \frac{a}{2} \norm{p}_{L^{2}(\Gamma)}^{2} \leq \norm{\Gamma_{\vec{v}}}_{L^{2}} \norm{p}_{L^{2}} + K \norm{(\mu + \chi \sigma) \nabla \varphi}_{L^{2}} \norm{\nabla p}_{L^{2}} + \frac{a}{2} \norm{g}_{L^{2}(\Gamma)}^{2} \\
& \quad \leq C \left ( 1 + \norm{g}_{L^{2}(\Gamma)}^{2} \right ) + \frac{K}{2} \norm{\nabla p}_{L^{2}}^{2} + \frac{a}{4} \norm{p}_{L^{2}(\Gamma)}^{2} + K \norm{(\mu + \chi \sigma) \nabla \varphi}_{L^{2}}^{2},
\end{align*}
which implies that
\begin{align}\label{Robin:nablap:Est}
\norm{p}_{H^{1}} \leq C \left ( 1 + \norm{g}_{L^{2}(\Gamma)} + \norm{(\mu + \chi \sigma) \nabla \varphi}_{L^{2}} \right ).
\end{align}
By the Gagliardo--Nirenburg inequality \eqref{GagNirenIneq} for $d = 3$, we see that
\begin{align}\label{GN:nablavarphi:L3}
\norm{\nabla \varphi}_{L^{3}} \leq C \norm{\varphi}_{H^{3}}^{\frac{1}{4}} \norm{\varphi}_{L^{6}}^{\frac{3}{4}},
\end{align}
and thus $(\mu + \chi \sigma) \nabla \varphi \in L^{\frac{8}{5}}(0,T;L^{2})$.  From \eqref{Robin:nablap:Est} this implies that 
\begin{align}\label{Robin:Est:3}
\norm{p}_{L^{\frac{8}{5}}(0,T;H^{1})} \leq C.
\end{align}
Analogous to \eqref{convection:term}, for $\xi \in L^{\frac{8}{3}}(0,T;H^{1}_{0})$, using that $\varphi \in L^{\infty}(0,T;H^{1}) \cap L^{2}(0,T;H^{3})$ and $\vec{v} \in L^{2}(Q)$ leads to
\begin{align*}
\abs{\int_{Q} \varphi \vec{v} \cdot \nabla \xi } \leq C \norm{\xi}_{L^{\frac{8}{3}}(0,T;H^{1}_{0})},
\end{align*}
which in turn gives
\begin{align}
\label{Robin:Est:4}
\norm{\pd_{t}\varphi}_{L^{\frac{8}{5}}(0,T;H^{-1})} \leq C
\end{align}
by the inspection of \eqref{Weak:R:varphi}. Similarly, by inspection of \eqref{Reg:weak:D:sigma}, the a priori estimate \eqref{Dirichlet:Est:5} is also valid.

The a priori estimates \eqref{Dirichlet:Est:5}, \eqref{Robin:Est:1}, \eqref{Robin:Est:2}, \eqref{Robin:Est:3} and \eqref{Robin:Est:4}, together with a Galerkin approximation are sufficient to deduce the existence of a quintuple $(\varphi^{\theta}, \mu^{\theta}, \sigma^{\theta}, \vec{v}^{\theta}, p^{\theta})$ satisfying the assertions of Lemma \ref{lem:Reg:problem:2}.  Once again, \eqref{Reg:problem:2:Est} follows from weak/weak* lower semi-continuity of the norms, and the assertion $0 \leq \sigma^{\theta} \leq 1$ a.e. in $Q$ follows from a weak comparison principle as in the proof of Lemma \ref{lem:Reg:problem:1}.
\end{proof}

\begin{remark}\label{Remark:WhyDirichlet}
The necessity of a Dirichlet condition for $\mu$ is due to the fact that we cannot control $\norm{\mu \nabla \varphi}_{L^{\frac{6}{5}}}$ in \eqref{Robin:Pressure:L2} simply with the left-hand side of \eqref{Aux:CHD:Robin:Est:2} if we assume $\pdnu \mu = 0$ on $\Sigma$.  One could consider the splitting
\begin{align*}
\norm{\mu \nabla \varphi}_{L^{\frac{6}{5}}} & \leq \norm{(\mu - \mean{\mu}) \nabla \varphi}_{L^{\frac{6}{5}}} + \abs{\mean{\mu}} \norm{\nabla \varphi}_{L^{\frac{6}{5}}} \leq \norm{\mu - \mean{\mu}}_{L^{6}} \norm{\nabla \varphi}_{L^{\frac{3}{2}}} + \abs{\mean{\mu}} \norm{\nabla \varphi}_{L^{\frac{6}{5}}} \\
& \leq C \norm{\nabla \mu}_{L^{2}} \norm{\nabla \varphi}_{L^{\frac{3}{2}}} + C \left ( 1 + \norm{\sigma - 1}_{L^{2}} + \norm{\Psi'(\varphi)}_{L^{1}} \right ) \norm{\nabla \varphi}_{L^{\frac{6}{5}}},
\end{align*}
and in order to control the second term, it is desirable to have an estimate of the form
\begin{align*}
\norm{\Psi'(\varphi)}_{L^{1}}^{2} \leq C \left ( 1 + \norm{\Psi(\varphi)}_{L^{1}}\right ).  
\end{align*}
This leads to the situation encountered in \cite{GLNeumann} and restricts $\Psi$ to have quadratic growth.  Furthermore, the ansatz in \cite{GLCHD,JWZ} is to consider the splitting
\begin{align*}
\abs{\int_{\Omega} \Gamma_{\vec{v}} (p - \mu \varphi) \dx} & = \abs{\int_{\Omega} \Gamma_{\vec{v}} (p - \mean{\mu} \varphi) + \Gamma_{\vec{v}} (\mean{\mu} - \mu) \varphi \dx} \\
&  \leq \abs{\int_{\Omega} \Gamma_{\vec{v}}(p - \mean{\mu} \varphi) \dx } + C \norm{\nabla \mu}_{L^{2}} \norm{\varphi}_{L^{2}}.
\end{align*}
If $p$ satisfies the Darcy law \eqref{Robin} with the boundary condition $\pdnu p = 0$ on $\Sigma$, and if $\Gamma_{\vec{v}}$ has zero mean, then we can write
\begin{align*}
p = (-\Laplace_{N})^{-1} \left ( \Gamma_{\vec{v}} /K  - \div ((\mu - \mean{\mu} + \chi \sigma) \nabla \varphi) - \mean{\mu} \div (\nabla (\varphi - \mean{\varphi})) \right ),
\end{align*}
where for $f \in L^{2}$ with $\mean{f} = \frac{1}{\abs{\Omega}} \int_{\Omega} f \dx = 0$, we denote $u := (-\Laplace_{N})^{-1}(f) \in H^{1}$ as the unique weak solution to
\begin{align*}
- \Laplace u = f \text{ in } \Omega, \quad \pdnu u = 0 \text{ on } \Gamma \text{ with } \mean{u} = 0.
\end{align*}
A short calculation shows that
\begin{align*}
-(-\Laplace_{N})^{-1} ( \div (\mean{\mu} \nabla (\varphi - \mean{\varphi}))) = \mean{\mu}(\varphi - \mean{\varphi}),
\end{align*}
and so
\begin{align*}
\int_{\Omega} \Gamma_{\vec{v}} (p - \mean{\mu} \varphi) \dx = \int_{\Omega} \Gamma_{\vec{v}} \left ( (-\Laplace_{N})^{-1} \left ( \Gamma_{\vec{v}}/K -  \div ((\mu - \mean{\mu} + \chi \sigma) \nabla \varphi) \right ) \right ) - \Gamma_{\vec{v}} \mean{\mu} \, \mean{\varphi} \dx.
\end{align*}
In \cite{GLCHD,JWZ}, $\Gamma_{\vec{v}}$ has zero mean, and so the last term on the right-hand side vanishes, but this is not the case in our present setting, and thus the approach of \cite{GLCHD,JWZ} seems not to give any advantage in deriving a priori estimates.
\end{remark}

\section{Neumann boundary conditions for the chemical potential}\label{sec:quadratic}
In this section, let us state an analogous result to Lemma \ref{lem:Reg:problem:2} for the regularized problem consisting of \eqref{Reg:1:Div}, \eqref{Reg:1:varphi}, \eqref{Reg:1:chem}, \eqref{Reg:1:sigma} and 
\eqref{Reg:2:Darcy}, but now we consider the boundary conditions
\begin{align}\label{Aux:Neumann:bc}
\pdnu \varphi = \pdnu \mu = 0, \quad K \pdnu p = a(g-p) \text{ on } \Sigma,
\end{align}
and \eqref{assump:quadratic} instead of \eqref{assump:Potential}.  The assertion is formulated as follows.

\begin{lemma}\label{lem:Reg:problem:3}
Under Assumption \ref{assump:Main} $($with \eqref{assump:quadratic} instead of \eqref{assump:Potential}$)$ and $0 \leq \sigma_{0} \leq 1$ a.e. in $\Omega$, for any $\theta \in (0,1]$, there exists a weak solution quintuple $(\varphi^{\theta}, \mu^{\theta}, \sigma^{\theta}, \vec{v}^{\theta}, p^{\theta})$ in the sense of Definition \ref{defn:Weaksoln:Neumann} with additionally $\sigma^{\theta} \in H^{1}(0,T;H^{-1})$, $\sigma^{\theta}(0) = \sigma_{0}$ a.e. in $\Omega$, and \eqref{Weak:D:sigma} is replaced by \eqref{Reg:weak:D:sigma}.  Furthermore, there exists a positive constant $C$ not depending on $\theta, \varphi^{\theta}, \mu^{\theta}, \sigma^{\theta}, \vec{v}^{\theta}, p^{\theta}$ such that \begin{equation}\label{Reg:problem:3:Est}
\begin{aligned}
& \norm{\Psi(\varphi^{\theta})}_{L^{\infty}(0,T;L^{1})} + \norm{\Psi'(\varphi^{\theta})}_{L^{2}(0,T;H^{1})} + \norm{\varphi^{\theta}}_{L^{\infty}(0,T;H^{1}) \cap L^{2}(0,T;H^{3})} \\
& \quad + \norm{\mu^{\theta}}_{L^{2}(0,T;H^{1})}  + \norm{\vec{v}^{\theta}}_{L^{2}(Q)} + \norm{p^{\theta}}_{L^{\frac{8}{5}}(0,T;H^{1})} + \norm{\pd_{t}\varphi^{\theta}}_{L^{\frac{8}{5}}(0,T;(H^{1})^{*})} \\
& \quad + \norm{\sigma^{\theta}}_{L^{2}(0,T;H^{1})} + \norm{\theta \pd_{t}\sigma^{\theta}}_{L^{2}(0,T;H^{-1})} + \norm{p^{\theta}}_{L^{2}(\Sigma)} \leq C.
\end{aligned}
\end{equation}
\end{lemma}

\begin{proof}
Once again we will only derive the a priori estimates and omit the details of the Galerkin approximation.  Substituting $\zeta = \mu + \chi \sigma$ into \eqref{Weak:N:varphi}, and upon adding with the equalities obtained from substituting $\xi = Z (\sigma - 1)$ in \eqref{Reg:weak:D:sigma}, $\zeta = \pd_{t}\varphi$ in \eqref{Weak:D:mu} and $\vec{y} = K^{-1} \vec{v}$ in \eqref{Weak:R:velo} we have
\begin{equation}\label{Aux:CHD:Neumann:Energy:est}
\begin{aligned}
& \frac{\dd}{\dt} \int_{\Omega} A \Psi(\varphi) + \frac{B}{2} \abs{\nabla \varphi}^{2} + \frac{Z}{2} \theta \abs{\sigma - 1}^{2} \dx  \\
& \qquad + \int_{\Omega} m(\varphi) \abs{\nabla \mu}^{2} + \frac{1}{K} \abs{\vec{v}}^{2} + Z \abs{\nabla \sigma}^{2} + Z h(\varphi) \abs{\sigma}^{2} \dx + \int_{\Gamma} a \abs{p}^{2} \dHaus \\
& \quad = \int_{\Omega} - \chi m(\varphi) \nabla \mu \cdot \nabla \sigma + \Gamma_{\varphi}(\mu + \chi \sigma) + \Gamma_{\vec{v}}(p - \varphi(\mu + \chi \sigma)) \dx \\
& \qquad + \int_{\Omega} Z h(\varphi) \sigma \dx + \int_{\Gamma} a g p \dHaus.
\end{aligned}
\end{equation}
For the terms $-\chi m(\varphi) \nabla \mu \cdot \nabla \sigma$ and $Z h(\varphi) \sigma$, as well as the boundary term $a gp$ on the right-hand side we use H\"{o}lder's inequality, Young's inequality and the Poincar\'{e} inequality applied to $(\sigma - 1)$ to obtain
\begin{align*}
& \abs{\int_{\Omega} - \chi m(\varphi) \nabla \mu \cdot \nabla \sigma + Z h(\varphi) (\sigma -1 + 1) \dx + \int_{\Gamma} a g p \dHaus} \\
& \quad \leq \frac{m_{0}}{4} \norm{\nabla \mu}_{L^{2}}^{2} + \left ( \frac{\chi^{2} m_{1}^{2}}{m_{0}^{2}} + 1 \right ) \norm{\nabla \sigma}_{L^{2}}^{2} + \frac{a}{2} \norm{p}_{L^{2}(\Sigma)}^{2} + \frac{a}{2} \norm{g}_{L^{2}(\Sigma)}^{2} + C(1 + Z^{2}).
\end{align*}
Since the pressure $p$ satisfies the same Poisson equation, by following the computations in Section \ref{sec:Robin} and the discussion in Remark \ref{Remark:WhyDirichlet}, we obtain 
\begin{align*}
\norm{p}_{L^{2}} & \leq C \left ( 1 + \norm{g}_{L^{2}(\Gamma)} + \norm{(\mu + \chi (\sigma-1 + 1)) \nabla \varphi}_{L^{\frac{6}{5}}} \right ) \\
&  \leq C \left ( 1 + \norm{g}_{L^{2}(\Gamma)} + \left ( \norm{\nabla \mu}_{L^{2}} + \norm{\nabla \sigma}_{L^{2}} \right ) \norm{\nabla \varphi}_{L^{\frac{3}{2}}} + ( 1 + \abs{\mean{\mu}}  )\norm{\nabla \varphi}_{L^{\frac{6}{5}}} \right ).
\end{align*}
Substituting $\zeta = 1$ in \eqref{Weak:D:mu}, we can estimate the mean of $\mu$ by
\begin{align}\label{Neumann:mean:mu}
\abs{\mean{\mu}} \leq C \left ( \norm{\sigma}_{L^{2}} + \norm{\Psi'(\varphi)}_{L^{1}} \right ),
\end{align}
and so by Young's inequality and the boundedness of $\Gamma_{\vec{v}}$, we see that
\begin{align*}
& \abs{X} := \abs{\int_{\Omega} \Gamma_{\vec{v}} (p - \varphi(\mu - \mean{\mu}) - \varphi( \mean{\mu} + \chi \sigma)) \dx} \\
& \quad \leq C \left ( \norm{p}_{L^{2}} + \norm{\varphi}_{L^{2}} \norm{\nabla \mu}_{L^{2}} + \left ( \norm{\sigma}_{L^{2}} +  \norm{\Psi'(\varphi)}_{L^{1}} \right )\norm{\varphi}_{L^{2}} \right ) \\
& \quad \leq C \left ( 1 + \norm{g}_{L^{2}(\Gamma)} + \left ( 1 + \norm{\nabla \mu}_{L^{2}} + \norm{\nabla \sigma}_{L^{2}} + \norm{\Psi'(\varphi)}_{L^{1}} \right ) \norm{\varphi}_{H^{1}}  \right ) \\
& \quad \leq \frac{m_{0}}{4} \norm{\nabla \mu}_{L^{2}}^{2} + \norm{\nabla \sigma}_{L^{2}}^{2} + C \left ( 1 + \norm{g}_{L^{2}(\Gamma)} + \norm{\Psi'(\varphi)}_{L^{1}}^{2} + \norm{\varphi}_{L^{2}}^{2} + \norm{\nabla \varphi}_{L^{2}}^{2} \right ).
\end{align*}
Using that $\Psi$ has quadratic growth, we can find positive constants $C_{4}$, $C_{5}$ such that
\begin{align*}
\abs{\Psi'(s)} \leq C_{4} \abs{s} + C_{5} \quad \forall s \in \R,
\end{align*}
and so by \eqref{assump:quadratic}
\begin{align}\label{Neumann:Psi':Psi}
\norm{\Psi'(\varphi)}_{L^{1}}^{2} \leq C \left ( 1 + \norm{\varphi}_{L^{2}}^{2} \right ) \leq C \left ( 1 + \norm{\Psi(\varphi)}_{L^{1}} \right ).
\end{align}
This implies that
\begin{align*}
\abs{X} \leq \frac{m_{0}}{4} \norm{\nabla \mu}_{L^{2}}^{2} + \norm{\nabla \sigma}_{L^{2}}^{2} + C \left ( 1 + \norm{g}_{L^{2}(\Gamma)}^{2} + \norm{\Psi(\varphi)}_{L^{1}} + \norm{\nabla \varphi}_{L^{2}}^{2} \right ).
\end{align*}
In a similar fashion, the second term on the right-hand side of \eqref{Aux:CHD:Neumann:Energy:est} can be estimated as
\begin{align*}
& \abs{\int_{\Omega} \Gamma_{\varphi} (\mu - \mean{\mu} + \mean{\mu} + \chi (\sigma -1  + 1)) \dx} \leq C \left ( 1 + \abs{\mean{\mu}} + \norm{\nabla \mu}_{L^{2}} + \norm{\nabla \sigma}_{L^{2}} \right ) \\
& \quad \leq \frac{m_{0}}{4} \norm{\nabla \mu}_{L^{2}}^{2} + \norm{\nabla \sigma}_{L^{2}}^{2} + C \left ( 1 + \norm{\Psi(\varphi)}_{L^{1}} \right ),
\end{align*}
and we obtain from \eqref{Aux:CHD:Neumann:Energy:est}
\begin{align*}
& \frac{\dd}{\dt} \int_{\Omega} A \Psi(\varphi) + \frac{B}{2} \abs{\nabla \varphi}^{2} + \frac{Z}{2} \theta \abs{\sigma - 1}^{2} \dx \\
& \qquad + \frac{m_{0}}{4} \norm{\nabla \mu}_{L^{2}}^{2} + \frac{1}{K} \norm{\vec{v}}_{L^{2}}^{2} + \left ( Z - \frac{\chi^{2} m_{1}^{2}}{m_{0}^{2}} - 3 \right ) \norm{\nabla \sigma}_{L^{2}}^{2} + \frac{a}{2} \norm{p}_{L^{2}(\Gamma)}^{2} \\
& \quad \leq C \left ( 1 + Z^{2} + \norm{g}_{L^{2}(\Gamma)}^{2} + \norm{\Psi(\varphi)}_{L^{1}} + \norm{\nabla \varphi}_{L^{2}}^{2} \right ).
\end{align*}
Applying Gronwall's inequality leads to
\eqref{Robin:Est}, and the a priori estimate \eqref{Robin:Est:1} follows by applying \eqref{Neumann:mean:mu},  \eqref{Neumann:Psi':Psi} and the Poincar\'{e} inequality \eqref{Poincare} for $\mu$ and $\sigma - 1$.  The other a priori estimates \eqref{Robin:Est:2}, \eqref{Robin:Est:3} follow from a similar argument.  For the time derivative $\pd_{t}\varphi$, we note that $\nabla \varphi \cdot \vec{v} \in L^{\frac{8}{5}}(0,T;(H^{1})^{*})$ by \eqref{GN:nablavarphi:L3}, and so from \eqref{Weak:N:varphi} it holds that
\begin{align}\label{Neumann:timederivative:est}
\norm{\pd_{t}\varphi}_{L^{\frac{8}{5}}(0,T;(H^{1})^{*})} \leq C.
\end{align}
Together with \eqref{Dirichlet:Est:5}, the a priori estimates \eqref{Robin:Est:1}, \eqref{Robin:Est:2}, \eqref{Robin:Est:3} and \eqref{Neumann:timederivative:est}, and a Galerkin approximation are sufficient to deduce the existence of a quintuple $(\varphi^{\theta}, \mu^{\theta}, \sigma^{\theta}, \vec{v}^{\theta}, p^{\theta})$ satisfying the assertions of Lemma \ref{lem:Reg:problem:3}.  Furthermore,  by weak/weak* lower semi-continuity of the norms we obtain the estimate \eqref{Reg:problem:3:Est}, and by a weak comparison principle, it also holds that $0 \leq \sigma^{\theta} \leq 1$ a.e. in $Q$.
\end{proof}
For the proof of Theorem \ref{thm:Neumann} we pass to the limit $\theta \to 0$, using the estimate \eqref{Reg:problem:3:Est}.  We omit the details as it is a standard argument.

\end{document}